\documentclass[11pt]{article}
\usepackage{latexsym,color,amsmath,amsthm,amssymb,amscd,amsfonts,mathtools,comment}

\usepackage[a4paper, left=1in, right=1in, top=1in, bottom=1in]{geometry}

\newtheorem{lemma}{Lemma}

\newtheorem{fact}[lemma]{Fact}
\newtheorem{theorem}[lemma]{Theorem}
\newtheorem{definition}[lemma]{Definition}

\newtheorem{lem}[lemma]{Lemma}
\newtheorem{prop}[lemma]{Proposition}

\newtheorem{rem}[lemma]{Remark}

\newtheorem*{remark*}{Remark}

\def\S{{\mathcal{S}}}

\def\eps{\varepsilon}

\def\qed{\hfill $\vcenter{\hrule height .3mm
		\hbox {\vrule width .3mm height 2.1mm \kern 2mm \vrule width .3mm
			height 2.1mm} \hrule height .3mm}$ \bigskip}

\def\eps{\varepsilon}

\def\to{\rightarrow}

\newcommand{\iprod}[2]{\left\langle #1,#2 \right\rangle} 

\def\RR{\mathbb{R}}

\def\NN{\mathbb{N}}

\def\Im{{\rm Im}\,}

\def\eps{{\varepsilon}}

\def\K{{\mathcal{K}}}

\newcommand{\setdef}[2]{\left\{ #1 \ \middle| \ #2 \right\}}

\title{A full classification of the isometries of the class of ball-bodies}
\author{Shiri Artstein-Avidan, Arnon Chor, Dan Florentin \thanks{The first and second named authors are supported in part by the ERC under the European Union’s Horizon 2020 research and innovation programme (grant agreement no. 770127), by ISF grant Number 784/20, and by the Binational Science Foundation (grant no. 2020329).}}
\date{\today}

\begin{document}
\maketitle 
\begin{abstract}
	Complementing our previous results, we give a classification of all isometries (not necessarily surjective) of the metric space consisting of ball-bodies, endowed with the Hausdorff metric. ``Ball bodies'' are    convex bodies which are intersections of translates of the Euclidean unit ball.   We show that any such isometry is either a rigid motion, or a rigid motion composed with the $c$-duality mapping. In particular, any isometry on this metric space has to be surjective. 
\end{abstract}

We denote the class of convex bodies (i.e. non-empty compact convex subsets of $\RR^n$) by $\K^n$.
The class $\S_n \subseteq \K^n$ consists of all convex bodies which are intersections of translated Euclidean unit balls, $x+B_2^n$. They can be equivalently defined as 
the summands of $B_2^n$, i.e. those $K \in \K^n$ for which there exists $L \in \K^n$ with $K+ L = B_2^n$, or as convex bodies all of whose sectional curvatures at every point are at least $1$. 
This class appears very naturally in several different subfields in convexity, for example in the study of bodies of constant width (see \cite{martini2019bodies} for a survey), in optimal transport with respect to special cost functions (see \cite{artstein2023zoo}), in the Kneser-Poulsen conjecture (see \cite{bezdek2008kneser}), in isoperimetric type questions (see \cite{borisenko2017reverse,drach2023reverse,bezdek2021volumetric}), in connection to the Hadwiger illumination conjecture (see \cite{bezdek2012illuminating}), and more.  
This class was studied in \cite{bezdek2007ball} where among other results, many analogues of combinatorial flavour such as Caratheodory, Helly and Steinitz type theorems were proved for this class.  
The first and third named authors have recently studied this class in detail as well \cite{AF-perrint}. Note however that in the notation of \cite{AF-perrint}, the class $\S_n$ also contains the ``trivial'' sets $\emptyset, \RR^n$, whereas in this note it does not.

The class $\K^n$ (and its subclass $\S_n$) is endowed with a natural metric, the Hausdorff metric $\delta$, given by
\[
\delta(K,T) = \inf \setdef{\lambda > 0}{K + \lambda B_2^n \supseteq T \text{ \  and \  } T + \lambda B_2^n \supseteq K}.
\]
The Hausdorff metric is a very useful and well known way to measure distance between convex sets, and the rigid motions are natural isometries of this metric space.  It was shown by Schneider~\cite{Schneider1975} that any surjective isometry of $(\K^n, \delta)$ is given by a rigid motion.

It is both surprising and trivial that the metric space $(\S_n, \delta)$ admits another, essentially different, isometry. Indeed, since the class $\S_n$ is given by summands of the ball $B_2^n$, one can map each $K \in \S_n$ to the unique $L \in \S_n$ such that $K-L = B_2^n$. This $L$ we denoted in \cite{AF-perrint, ACF-preprint} by $K^c$ (since it can also be seen as a cost-induced transform in the terminology of \cite{artstein2023zoo}). Note that with these definitions $h_K(u) + h_{K^c}(-u) = 1$ for all $u \in S^{n-1}$, where $h_K$ is the support function of $K$ given by $h_K(u) = \max_{x \in K} \iprod{x}{u}$. Therefore, since the Hausdorff metric $\delta$ coincides with the $L^\infty$ metric on the class of support functions restricted to the unit sphere $S^{n-1}$, the map $K \mapsto K^c$ is an isometry with respect to $\delta$. 

In~\cite{AF-perrint} we showed that, similarly to the result of Schneider for $\K^n$ \cite{Schneider1975}, the only surjective isometries of $\S_n$ are given, up to rigid motions,  either by the identity or by $K \mapsto K^c$.  
We thus see that the class $\S_n$ admits not one, but {\em two} essentially different isometries. It is instructive to note that the isometry $K\mapsto K^c$ satisfies that for any $K \in \S_n$ we have $K^{cc} = K$, and also 
\[
    K^c = \bigcap_{x \in K} (x + B_2^n).
\]
Moreover, it is easy to see that the $c$-duality mapping 
respects Minkowski averaging (for a discussion see \cite{AF-perrint}), namely 
for  $K,T\in \S_n$ and $\lambda \in (0,1)$, one has
\[ ((1-\lambda)K + \lambda T)^c = (1-\lambda)K^c + \lambda T^c.\]
We mention in passing that it was also shown in~\cite{AF-perrint} that  $K\mapsto K^c$  is, again up to rigid motions, the unique order-reversing bijection on $\S_n$.

Going back to $(\K^n, \delta)$, Gruber and Lettl~\cite{GruberLettl1980} extended Schneider's result, and showed that in $\K^n$, without the surjectivity assumption, the only isometries which exist are of the form $K\mapsto gK+L$ for a rigid motion $g$ and a fixed $L\in \K^n$.  In this note we provide the analogous result for the metric space   $(\S_n,\delta)$. However, it turns out that the only isometries which exist are in fact surjective, and are thus given either by a rigid motion or by a rigid motion composed with the $c$-duality. This is yet another key difference between the class $\S_n$ and the class of all convex bodies.  

\begin{theorem}\label{thm:main_not_bijective2}
Let $T:\S_n\to \S_n$ be an isometry (not assumed a-priori to be surjective) with respect to the Hausdorff distance. Then there exists a rigid motion (orthogonal transformation and translation) $g:\RR^n \to \RR^n$ such that either $TK = gK$ for all $K\in \S_n$ or $TK = gK^c$ for all $K\in \S_n$. 
\end{theorem}

The method of proof is completely different from the method we used in \cite{ACF-preprint}, which relied on understanding special geodesics of the metric space, and so was, in some sense, a more ``local'' approach. 
The idea of our proof in this note is as follows. We will first use topological arguments to show that there is a point (that is, a singleton $\{x_0\}$, which for simplicity we will denote by $x_0$) in the image of $T$. In fact, we will show that this point is the image of either a unit ball or a point. We will then use an argument very similar to that in \cite{gruber1980isometrien} to show that in fact all points are in the image of the isometry $T$, and are the images of translates of one given set, which is either a point or a Euclidean unit ball.
The rest of the argument is quite standard - the isometry condition will then imply that the map restricted to these translates must be induced by a rigid motion, and since a set in $\S_n$ can be reconstructed using its distance from points, the proof will be complete. 

To prove the first step, namely that some point must be in the image of $T$, we will use a very natural fact, which we were not able to locate in the literature and thus include its proof at the end of this note. To state it, we define what is an $\varepsilon$-isometry in a metric space (see e.g.~\cite{burago2001course}). 
\begin{definition}
    Let $\varepsilon \geq 0$. A map $f: X \to Y$ between two metric spaces $(X,d_X), (Y,d_Y)$ is called an $\varepsilon$-isometry if for all $x,x' \in X$,
    \[
        \big| d_Y(f(x), f(x')) - d_X(x,x') \big| \leq \varepsilon.
    \]
\end{definition}
These types of near-isometries (and their generalizations such as quasi-isometries) are usually used to study the large-scale geometry of a space in a way which ignores small-scale details. For example, see the pioneering works of Gromov \cite{gromov1999metric}, or \cite{vaisala2002survey} for a survey. In particular, they are seldom assumed to be continuous.  Many results are known about these objects, for example that if $X = Y = \RR^n$ then any $\eps$-isometry is  ``coarsely surjective'', namely any point in $\RR^n$ is within some fixed distance of the image of $f$ (see \cite[Exercise 6.12]{inbook}).
However, it turns out that  with the assumption of 
continuity, and when considering the Euclidean space $\RR^n$,  an $\eps$-isometry must be surjective. The following lemma is probably known, but we have not found it in the literature and thus prove it at the end of this note. 

\begin{lemma}\label{lem:quasi_is_onto}
    Let $\varepsilon \geq 0$ and let $f: \RR^n \to \RR^n$ be a continuous $\varepsilon$-isometry. Then $f$ is surjective.
\end{lemma}

The main step in the proof is to show that for an isometry $T$ of $(\S_n, \delta)$, there must be either a point or a unit ball which is mapped under $T$ to a point. The idea of the proof is as follows. First we use Lemma~\ref{lem:quasi_is_onto} to find a point and a Euclidean unit ball whose images under $T$ have the same circumcenter. We then note that on the one hand, these images must have Hausdorff distance at least 1 from each other, since any point and any unit ball are at a distance of at least $1$, and on the other hand, two sets in $\S_n$ with the same circumcenter have distance at most $1$, and when the distance is $1$ then one of them must be a point. We conclude that if two bodies in $\S_n$ have the same circumcenter and are at distance at least $1$, then one of them must be a point.

\begin{prop}\label{prop:point_in_image}
    Let $T: \S_n \to \S_n$ be an isometry with respect to the Hausdorff distance. Then either there exists a point $x_0 \in \RR^n$ such that $T(x_0)$ is a point, or there exists 
    a Euclidean unit ball $y_0 + B_2^n$ such that $T(y_0+B_2^n)$ is a point.
\end{prop}

\begin{proof}[Proof of Proposition~\ref{prop:point_in_image}]
Denote by $c: \S_n \to \RR^n$ the circumcenter map, i.e. 
$c(K)$ is the center of its (unique) outball, and denote by $r: \S_n \to [0,1]$ the function that returns the circumradius of $K$. Note that by convexity, $c(K) \in K$ for $K \in \K^n$. The map $c$ is continuous, see for example \cite[Exercise 9.2.25]{burago2001course}.
We start by showing  that there is a point $x_0 \in \RR^n$ and a unit ball $y_0 + B_2^n$ for some $y_0 \in \RR^n$ such that $c(T(x_0)) = c(T(y_0 + B_2^n))$.

To this end,  define two functions $f_{pt},f_{ball}: \RR^n \to \RR^n$ by
\[
  f_{pt}(x) = c(T(x)) \qquad{and}\qquad f_{ball}(x) = c(T(x+B_2^n)) .
\]
Note that $\delta(K,c(K)) \leq 1$ for any $K \in \S_n$, since $c(K) \in K$ and $K \subseteq c(K) + B_2^n$. Thus for any $K, L \in \S_n$, by the triangle inequality we 
have $|\delta(K,L) - \delta (c(K),c(L))| \le 2$. Since  $|x-y| = \delta (x,y) = \delta (T(x), T(y))$ for $x,y \in \RR^n$, applying that inequality to the bodies $K=T(x)$ and $L=T(y)$ yields $\big| |x-y| - |f_{pt}(x) - f_{pt}(y)| \big| \leq 2$. Similarly, since $|x-y| = \delta (x+B_2^n,y+B_2^n) = \delta (T(x+B_2^n), T(y+B_2^n))$, we have $\big| |x-y| - |f_{ball}(x)-f_{ball}(y)| \big| \le 2$, i.e. $f_{pt}$ and $f_{ball}$ are $2$-isometries, and they are also continuous as compositions of continuous maps. Lemma~\ref{lem:quasi_is_onto} implies that both $f_{pt}$ and $f_{ball}$ are surjective, and in particular for any $c_0\in \RR^n$ there exist $x_0,y_0 \in \RR^n$ such that $c(T(x_0)) = f_{pt}(x_0) = c_0 = f_{ball}(y_0) = c(T(y_0+B_2^n))$. Denote $r_1 = r(T(x_0))$ and $r_2 = r(T(y_0 + B_2^n))$.

Next, note that if $x + tB_2^n \supseteq y + B_2^n$ then $t \geq 1$, thus for any $x,y \in \RR^n$ we have 
\begin{align}\label{al:1}
 \delta(x,y+B_2^n) \geq 1, 
\end{align}
with equality if and only if $x = y$. We claim $\{r_1,r_2\}=\{0,1\}$. Indeed, if $\max \{r_1,r_2\} < 1$ then
\[
        T(x_0) \subseteq c_0 + r_1 B_2^n \subseteq T(y_0 + B_2^n) + r_1 B_2^n \qquad{and}\qquad T(y_0 + B_2^n) \subseteq c_0 + r_2 B_2^n \subseteq T(x_0) + r_2 B_2^n ,
\]
which implies $\delta(x_0,y_0+B_2^n) = \delta(T(x_0),T(y_0+B_2^n)) \leq \max \{r_1,r_2\} < 1$, which is a contradiction to \eqref{al:1}. Thus $\max\{r_1,r_2\} = 1$. Assume $r_1 = 1$. The only sets in  $\S_n$ with out-radius $1$ are translates of Euclidean balls, i.e. $T(x_0) = c_0 + B_2^n$. If $T(y_0 + B_2^n)$ is not a point, it contains its circumcenter $c_0$ in its interior, meaning there is a small ball $c_0 + tB_2^n \subseteq T(y_0+B_2^n)$, but then $T(x_0) = c_0 + B_2^n \subseteq T(y_0 + B_2^n) + (1-t)B_2^n$ and $T(y_0 + B_2^n) \subseteq T(x_0)$, which implies $\delta(T(x_0),T(y_0+B_2^n)) \leq 1-t < 1$, again a contradiction. Therefore $c_0=T(y_0 + B_2^n)$ is a point in the image of $T$.
The case $r_2 = 1$ is worked out similarly, to obtain $c_0=T(x_0)$.

\end{proof}

Our second ingredient for the proof of Theorem \ref{thm:main_not_bijective2} has to do with a simple property of geodesics in $(\K^n, \delta)$. We recall that 
in \cite{ACF-preprint} we  used the fact that the only geodesic connecting two points $x,y$ in $\S_n$ is given by $(1-\lambda)x+\lambda y$. Here we will make use of an ``opposite'' fact, which is that a point $x$ cannot lie in the interior of a geodesic connecting two bodies in $\S_n$ which are themselves not points. We mention that in general points {\em can} participate in geodesics which are not entirely made of points, for example $(K_t)_{t\in [0,1]}$ given by $K_t = tu$ for $t\in \left[0,\frac12 \right]$ and $K_t = \frac12 u + \left(t-\frac12 \right) B_2^n$ for $t\in \left[ \frac12, 1 \right]$ is a geodesic in $\S_n$, for any $u \in S^{n-1}$. 

\begin{lemma}\label{lem:not_body_point_body}
    Let $K_0, K_1, K_2 \in \K^n$ such that $\delta(K_0,K_1) + \delta(K_1,K_2) = \delta(K_0,K_2)$. If $K_0,K_2$ are not points, then $K_1$ is also not a point. 
\end{lemma}

\begin{proof}
Assume that $K_0,K_2$ are not points, and suppose that $K_1 = \{p\}$ is a point. Then there exist points $x,y \in K_0$ such that $|p-x| < |p-y| = \delta(K_0,K_1)$, and similarly for $K_2$, there exist points $w,z \in K_2$ such that $|p-z| <|p-w| = \delta(K_1,K_2)$. Note that
\[
K_2 \subseteq
p + |p-w|B_2^n \subseteq
K_0 + (|p-x| + |p-w|)B_2^n  ,
\]
and similarly
\[
K_0 \subseteq
p + |p-y|B_2^n \subseteq 
K_2 + (|p-z| + |p-y|)B_2^n.
\]
Therefore
\begin{eqnarray*}
\delta(K_0,K_2) &\le&
\max \left\{|p-x|+|p-w|, |p-y|+|p-z|\right\} \\
&<& |p-y|+|p-w| = \delta(K_0,K_1) + \delta(K_1,K_2),
\end{eqnarray*}
in contradiction to the hypothesis.
\end{proof}

As the last element of the proof of Theorem \ref{thm:main_not_bijective2} 
we recall the following simple lemma (given as \cite[Lemma 13]{ACF-preprint}) which says that a body is specified by its Hausdorff distances to points. 

\begin{lem}\label{lem:representations}
For any $K\in \S_n$, one has
\begin{eqnarray*}
K = \bigcap_{x\in \RR^n} (x + \delta (x, K) B_2^n).
\end{eqnarray*}
\end{lem}

We also make use of the following fact which states that isometries between (connected subsets with nonempty interior of) Euclidean spaces are affine. This fact is folklore, see for example~\cite[Proposition 9.1.3]{berger2009geometry} for a proof of the case where the isometry is $\RR^n \to \RR^n$. The case where the domain is a connected subset of $\RR^n$ with nonempty interior has an identical proof.

\begin{fact}\label{fact:isometry_is_affine}
    Let $A \subseteq \RR^n$ be connected with nonempty interior, and let $f: A \to \RR^n$ be an isometry. Then there is an affine map $F: \RR^n \to \RR^n$ such that $f = F|_A$.
\end{fact}


We are now in a position to prove our main  theorem.

\begin{proof}[Proof of Theorem~\ref{thm:main_not_bijective2}]
    By Proposition~\ref{prop:point_in_image}, there is some $K_0 \in \S_n$ such that $TK_0$ is a point, and furthermore $K_0$ is either a point or a Euclidean unit ball. Denote 
    \[
        M = \setdef{x \in \RR^n}{T(K_0 + x) \text{ is not a point}} .
    \]
We show $M$ is convex. Let $p,q \in M$ and $r \in [p,q]$. Thus 
$\delta(K_0 + p,K_0 + r) + \delta(K_0 + r,K_0 + q) = 
\delta(p,r) + \delta(r,q) = |p-r| + |r-q| = |p-q| = \delta(p,q) = \delta(K_0 + p, K_0 + q)$. Since $T$ is an isometry, $\delta(T(K_0 + p), T(K_0 + r)) + \delta(T(K_0 + r), T(K_0 + q)) = \delta(T(K_0 + p), T(K_0 + q))$. Since $T(K_0 + p)$ and $T(K_0 + q)$ are not points, Lemma~\ref{lem:not_body_point_body} implies that neither  is $T(K_0 + r)$, proving the convexity of $M$. Next we show $M$ must in fact be empty. 
Since $TK_0$ is a point, $0\notin M \neq \RR^n$, and so, seeing that $M$ is convex, there is a closed half-space $H$ such that $M \cap H = \emptyset$, i.e. for any $x \in H$, $Tx$ is a point. The restriction of $T$ to points in $H$ is an isometry from a closed half-space to $\RR^n$, and so by Fact~\ref{fact:isometry_is_affine} is induced by a rigid motion. By composing with the inverse rigid motion we may assume without loss of generality that for any $x \in H$, $Tx = x$. 
    
Suppose towards a contradiction that $M$ is not empty and fix some $p\in M$. For any $x \in \partial H$, $\delta(Tp,x) = \delta(Tp,Tx) = \delta(p,x)$, which implies $Tp \subseteq x + \delta(p,x)B_2^n$. Intersecting over all $x \in \partial H$,
    \[
        Tp \subseteq \bigcap_{x \in \partial H} x + \delta(p,x)B_2^n = [p,p'] , 
    \]
    where $p'$ is the reflection of $p$ through $\partial H$. Since $Tp \in \S_n$ and is contained in a segment, it must be a point, in contradiction to the assumption $p \in M$. Thus $M = \emptyset$.

    So far we have obtained that for any $x \in \RR^n$, $T(K_0 + x)$ is a point. If $K_0$ is a point itself, $T$ is an isometry that maps points to points and therefore by Fact~\ref{fact:isometry_is_affine} its restriction to points must be a rigid motion $g: \RR^n \to \RR^n$. Thus $Q = g^{-1} \circ T: \S_n \to \S_n$ is an isometry restricting to the identity on points. Let $K \in \S_n$. Note that for any $x \in \RR^n$ we have $\delta(QK,x) = \delta(K,x)$, so using Lemma~\ref{lem:representations} twice, we get
    \[
        QK = \bigcap_{x \in \RR^n}(x + \delta(x,QK)B_2^n) = \bigcap_{x \in \RR^n}(x + \delta(x,K)B_2^n) = K ,
    \]
    so $K = QK = g^{-1}TK$, which implies $TK = gK$. 

If $K_0$ is a unit ball, consider $R: \S_n \to \S_n$ given by $RK = TK^c$. This is an isometry as a composition of two isometries. Since $T$ maps unit balls to points, $R$ maps points to points, and thus by Fact~\ref{fact:isometry_is_affine} its restriction to points must be a rigid motion $g: \RR^n \to \RR^n$. Thus $Q = g^{-1} \circ R: \S_n \to \S_n$ is an isometry restricting to the identity on points. As before we deduce that $K=QK$ for all $K \in \S_n$, and since $K = QK = g^{-1}TK^c$, we get $TK = gK^c$. This completes the proof of our main theorem. 
\end{proof}

We end this note with a proof of Lemma \ref{lem:quasi_is_onto}, which we were not able to locate in the literature. 
We will use the following immediate special case of~\cite[Theorem 3.3]{alestalo2001isometric}, which states that $\varepsilon$-isometries on a compact set which is ``not too flat'' are $L^\infty$-close to some isometry.


\begin{prop}\label{lem:almost_affine}
   For any $n \in \NN$ there exists $C_n > 0$ such that the following holds. Let $\varepsilon, R > 0$.  Then for any $\varepsilon$-isometry $f: RB_2^n \to \RR^n$  there exists an affine isometry $U: \RR^n \to \RR^n$ such that for any $x \in RB_2^n$ we have
    \[
        |f(x) - U(x)| \leq C_n \varepsilon = :C .
    \]
\end{prop}

\begin{rem}
    We do not need the exact dependence of the constant $C$ on $n$ and $\eps$, but we do make use of the fact that it does not depend on $R$. 
\end{rem}

The proof of Lemma \ref{lem:quasi_is_onto} makes use of the notion of the degree of a map $S^{n-1}\to S^{n-1}$, and the fact that it is invariant under homotopies. For reference, see e.g.~\cite[\textsection 2.2]{hatcher2002algebraic} or \cite[\textsection 5]{milnor1997topology}. 

\begin{proof}[Proof of Lemma \ref{lem:quasi_is_onto}]
Let $f: \RR^n \to \RR^n$ be a continuous $\varepsilon$-isometry, and suppose towards a contradiction that there is some $y \in \RR^n\setminus \Im f$. Fix some $R > |y-f(0)| + C + \eps  $, where $C > 0$ is the constant given by Proposition~\ref{lem:almost_affine}, and let $U: \RR^n \to \RR^n$ be the affine isometry given by applying Proposition~\ref{lem:almost_affine} to $f|_{RB_2^n}$. Since $y \not\in \Im f$, we may define $\tilde{f}: RB_2^n \to S^{n-1}$  by
    \begin{align*}
        \tilde{f}(x)=\frac{f(x)-y}{|f(x)-y|}.
    \end{align*}
This mapping is well defined and continuous. Note that $\tilde{f}$ is defined on the whole ball $RB_2^n$ which is contractible, thus $\tilde{f}|_{RS^{n-1}}$ is null-homotopic and in particular $\deg \tilde{f}|_{RS^{n-1}} = 0$. However, $U(RB_2^n)$ is a ball of radius $R$ centered at $U(0)$, and it contains $y$ in its inerior, since
\[| y-U(0)| \leq |y-f(0)| + |f(0) - U(0)| \leq |y - f(0)| + C < R.\]
We may thus define $\tilde{U}: RS^{n-1} \to S^{n-1}$ by
    \[
        \tilde{U}(x)=\frac{U(x) - y}{|U(x)-y|} .
    \]
Since $y$ is in the interior of the ball $U(RB_2^n)$, any ray emanating from $y$ intersects $\partial U(RB_2^n)$ exactly once, i.e. $\tilde{U}$ is bijective, and in particular 
$\deg \tilde{U} \neq 0$. For any $x \in RS^{n-1}$, by the triangle inequality
\[
  |f(x) - y| \geq \big| |f(x)-f(0)| - |f(0) - y| \big| \ge (R-\varepsilon) - |y-f(0)| > C ,
\]
and so $y \not\in f(x) + C B_2^n$ and by choice of $U$ we know $U(x) \in f(x) + C  B_2^n$. Therefore, for any $x\in RS^{n-1}$, $y$ does not belong to the line segment $[f(x),U(x)]$ connecting $f(x)$ and $U(x)$, so we may define a homotopy $h: RS^{n-1} \times [0,1] \to S^{n-1}$ between $\tilde{f} |_{RS^{n-1}}$ and $\tilde{U}$, by
\[
h(x,t) \mapsto \frac{(1-t)f(x) + tU(x) - y}{|(1-t)f(x) + tU(x) - y|}.
\]
Since the maps $\tilde{f} |_{RS^{n-1}}$ and $\tilde{U}$ are homotopic, we conclude that $0=\deg \tilde{f}|_{RS^{n-1}} = \deg \tilde{U}$, a contradiction. 
\end{proof}

\bibliographystyle{plain}
\bibliography{ref}

{\small
\noindent S. Artstein-Avidan and A. Chor, 
\vskip 2pt
\noindent School of Mathematical Sciences, Tel Aviv University, Ramat
Aviv, Tel Aviv, 69978, Israel.\vskip 2pt
\noindent Email: shiri@tauex.tau.ac.il, arnonchor@mail.tau.ac.il.
\vskip 2pt
\noindent D.I. Florentin, 
\vskip 2pt
\noindent Department of Mathematics, Bar-Ilan University, Ramat Gan,  52900, Israel.   \vskip 2pt
\noindent Email: dan.florentin@biu.ac.il.
}

\end{document}